\documentclass[12pt]{article}
\usepackage[english]{babel}
\usepackage{amsthm}
\usepackage{amsmath}
\usepackage{amssymb}
\usepackage{graphicx}
\usepackage{caption}
\usepackage{subcaption}
\usepackage{color}
\usepackage[a4paper,left=15mm,right=15mm,top=30mm,bottom=20mm]{geometry}
\usepackage{booktabs}

\newcounter{logcounter}[section]

\newcounter{logsubcounter}[logcounter]

\theoremstyle{plain}

\theoremstyle{definition}

\newtheorem{Th}{Theorem}
\newtheorem{Lm}{Lemma}
\newtheorem{Def}{Definition}

\newcommand*{\BCf}[2]{{{#1}\choose{#2}}}

\begin{document}

\bibliographystyle{unsrt}

\title{Counting Unlabelled Chord Diagrams of Maximal Genus}

\author{
Evgeniy Krasko\\
\small Saint Petersburg Academic University\\[-0.8ex]
\small\tt krasko.evgeniy@gmail.com
}

\date{{Aug 20, 2017}{}}

\begin{abstract}
Maximal chord diagrams up to all isomorphisms are enumerated. The enumerating formula is based on a bijection between rooted one-vertex one-face maps on locally orientable surfaces and a certain class of symmetric chord diagrams. This result extends the one of Cori and Marcus regarding maximal chord diagrams enumerated up to rotations.

\end{abstract}

\maketitle

\section*{Introduction}
A {\em chord diagram} is a circle with $2n$ {\em points} on its circumference joined pairwise by {\em chords} (Figure \ref{fig:basic}(a)). From the combinatorial point of view chord diagrams are analogous to one-face maps on orientable surfaces of some genus: any chord diagram defines a way to glue the sides of a $2n$-gon that yields an oriented surface with an one-face map formed by edges and vertices of the polygon embedded into it (Figure \ref{fig:basic}(b)). Each pair should be glued without a twist, as shown by arrows on Figure \ref{fig:basic}(b). Only the relative orientation of the arrows in each pair matters. 

The notion of a {\em dual map} allows to interchange faces and vertices, hence chord diagrams are isomorphic to one-vertex maps too. The {\em genus} of a diagram is the one of the corresponding map.

Chord diagrams and their properties have been studied by many authors. A classical enumerative result states that the number of planar (genus 0) chord diagrams with $n$ chords is equal to the Catalan number $C_n = \frac{1}{n+1}\BCf{2n}{n}$. Walsh and Lehman have proved \cite[(14)]{Walsh_Lehman} that in the other extreme case, when an one-face map with $2g$ edges is required to be of maximal possible genus, $g$, the enumerating formula is also remarkably simple: $\mu_{2g} = \frac{(4g)!}{4^g(2g+1)!}$. The formula of Harer and Zagier \cite[Theorem 2]{HarerZagier} allows to find the number of chord diagrams for a given genus $g$ and a given number of edges $n$.

Cori and Marcus named genus $g$ diagrams with $2g$ chords {\em maximal} and found an enumerating formula for the number of isomorphism classes (often referred to as {\em unlabelled diagrams}) of them \cite{Cori_Marcus}. The notion of isomorphism used in that work is restricted to rotations of the circle on which a diagram is drawn.

It is natural to try to extend this result and enumerate isomorphism classes of maximal chord diagrams under the action of the largest possible group of symmetries, the dihedral group $D_n$. This problem could be rephrased as the problem of enumerating genus $g$ maps which have both one face and one vertex, up to orientation-preserving and orientation-reversing isomorphisms. 

Not so many results regarding map enumeration up to orientation-reversing isomorphisms are known. There exist a formula \cite{Read_dissections} for planar polygon dissections (isomorphic to trees), two different approaches \cite{Wormald} \cite{Liskovets_reductive} to enumerate arbitrary spherical maps, but not much has been done so far to generalize these results to higher genera. In this note we will solve the problem of enumerating maximal chord diagrams of genus $g$ up to all isomorphisms and thus solve the genus-$g$ map enumeration problem for the corresponding class of maps.

\begin{figure}[h]
\centering
\begin{tabular}[t]{cc}
	\begin{subfigure}[b]{0.22\textwidth}
	\centering
    		\includegraphics[scale=1.9]{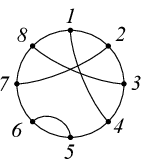}
 	\caption{A chord diagram}
	\end{subfigure}

	\begin{subfigure}[b]{0.22\textwidth}
	\centering
    		\includegraphics[scale=1.9]{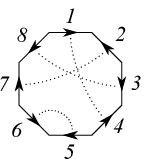}
 	\caption{An octagon}
	\end{subfigure}

	\begin{subfigure}[b]{0.22\textwidth}
	\centering
    		\includegraphics[scale=1.9]{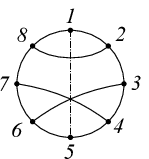}
 	\caption{A type I diagram}
	\end{subfigure}

	\begin{subfigure}[b]{0.22\textwidth}
	\centering
    		\includegraphics[scale=1.9]{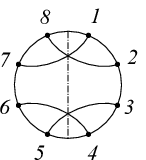}
 	\caption{A type II diagram}
	\end{subfigure}
\end{tabular}
\caption{Chord diagrams and polygon gluings}
\label{fig:basic}
\end{figure}

\section*{Basic facts and definitions}
We begin with recalling the {\em Burnside's lemma} that states that the number $|X/G|$ of isomorphism classes of objects in a set $X$ under the action of a group $G$ is given by
$$
|X/G|  = \frac{1}{|G|} \sum_{g \in G} |\textrm{Fix}_g(X)|,
$$
$\textrm{Fix}_g(X)$ being the subset of elements of $X$ fixed by the action of $g$. It follows that for a given class $\mathcal{D}$ of chord diagrams with $n$ edges, the numbers $d^*_n$ of diagrams counted up to rotations only and $d^{\circ}_n$ of diagrams counted up to all isomorphisms can be expressed as follows:
$$
d^*_n = \frac{1}{2n} \sum_{g \in C_{2n}} |\textrm{Fix}_g(\mathcal{D})|; \quad d^{\circ}_n = \frac{1}{4n} \sum_{g \in D_{2n}} |\textrm{Fix}_g(\mathcal{D})|; \qquad \Rightarrow \qquad d^{\circ}_n = \frac{d^*_n}{2} + \frac{1}{4n} \sum_{g \in D_{2n} \backslash C_{2n}} |\textrm{Fix}_g(\mathcal{D})|.
$$
Here the groups $C_{2n}$ and $D_{2n}$ are the cyclic group of rotations and the dihedral group of all symmetries of a circle with $2n$ evenly distributed points, correspondingly. The elements of $D_{2n} \backslash C_{2n}$ are just reflections of two types: $n$ reflections with respect to an axis that passes through two opposite points, and $n$ reflections with respect to an axis that passes through the centers of two opposite arcs into which the circumference is divided. Since the class $\mathcal{D}$ of maximal diagrams that we will be interested in is closed under rotations, for any two reflections $g$ and $h$ of the same type it is true that $|\textrm{Fix}_g(\mathcal{D})| = |\textrm{Fix}_h(\mathcal{D})|$. Hence for maximal diagrams
\begin{equation}
\label{final}
d^{\circ}_{2g} = \frac{d^*_{2g}}{2} + \frac{d^|_{2g}+d^{||}_{2g}}{4},
\end{equation}
$d^|_{2g}$ being the number of diagrams with the axis of symmetry passing through the points $1$ and $2g+1$ ({\em type I} diagrams, Figure \ref{fig:basic}(c)), $d^{||}_{2g}$ being the number of diagrams with the axis of symmetry passing through the centers of the arcs $(4g) - (1)$ and $(2g) - (2g+1)$ ({\em type II} diagrams, Figure \ref{fig:basic}(d)).
For maximal diagrams an explicit formula for the numbers  $d^*_{2g}$ is given in \cite{Cori_Marcus}. So from now on, we will focus on finding the expressions for the numbers $d^|_{2g}$ and $d^{||}_{2g}$. First we have to introduce a few definitions.
 \begin{Def}
For a given chord diagram, a {\em face-walk} is a cyclic sequence consisting of alternating chord sides and arcs of the circumference obtained by traversing this diagram as shown on Figure \ref{fig:def}(a).
 \end{Def}

If we represent a chord diagram as an one-face map obtained by polygon gluing, each face walk would correspond to a traversal of all semi-edges incident to one vertex, in a cyclic order. In a dual one-vertex map the same face-walk would correspond to a traversal of a single face. It is a routine to check that maximal diagrams are those that contain exactly one face-walk.
 
 \begin{Def}
For a given reflection and a chord diagram, a {\em vertical chord} is a chord that lies on its axis of symmetry. A {\em horizontal chord} is a chord whose ends are interchanged by the reflection.
 \end{Def}

For an one-face-map representation of a reflective-symmetric chord diagram one can use the concept of a {\em quotient map} introduced by Liskovets \cite[Section 4]{Liskovets_reductive}.
 
  \begin{Def}
Let $S$ be a genus $g$ surface together with a map embedded into it, and let $\alpha$ be an automorphism of this embedding (Figure \ref{fig:def}(b)). The corresponding {\em quotient map} is the result of identifying the points of each orbit of the action of $\alpha$ on $S$ (Figure \ref{fig:def}(c)).
 \end{Def}
 
 Note that to obtain the correct direction of gluing depicted on the Figure \ref{fig:def}(c) we should examine each pair of chords that corresponds to gluing a pair of sides of the polygon: if these chords intersect, the sides will be glued with a twist; if not, the sides will be glued in a forward direction. If we actually glue the sides of the polygon depicted on Figure \ref{fig:def}(c) taking the arrow directions into account, we will obtain a map that can be drawn as on Figure \ref{fig:def}(d). If at least one pair of sides is glued with a twist, the embedding surface of the map becomes non-orientable. Liskovets \cite{Liskovets_reductive} shows that the resulting quotient map can be viewed as a certain generalization of a map and notes its peculiarities. Some of these peculiarities can be seen on Figure \ref{fig:def}(d): the surface is a cylinder with a hole capped with a M\"{o}bius band, one of its boundaries contains an edge and a vertex of the original map, while the other contains a vertex only.

\begin{figure}[h]
\centering
\begin{tabular}[t]{cc}
	
	\begin{subfigure}[b]{0.22\textwidth}
	\centering
    		\includegraphics[scale=1.9]{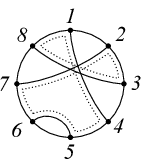}
 	\caption{Face walk}
	\end{subfigure}
	
	\begin{subfigure}[b]{0.22\textwidth}
	\centering
    		\includegraphics[scale=1.9]{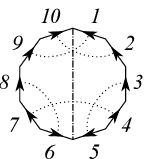}
 	\caption{Symmetric map}
	\end{subfigure}
	
	\begin{subfigure}[b]{0.22\textwidth}
	\centering
    		\includegraphics[scale=1.9]{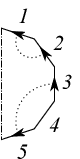}
 	\caption{Quotient map}
	\end{subfigure}

	\begin{subfigure}[b]{0.22\textwidth}
	\centering
    		\includegraphics[scale=1.9]{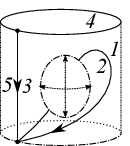}
 	\caption{Map on a surface}
	\end{subfigure}
	
\end{tabular}
\caption{Definitions}
\label{fig:def}
\end{figure}

 \section*{The structure of reflective-symmetric maximal diagrams}
  
It turns out that for maximal chord diagrams we can avoid most of the quotient map peculiarities described in \cite{Liskovets_reductive}. To formalize that, we first need the following auxiliary lemmata.
 
 \begin{Lm} 
 \label{lm:1}
 No type II maximal diagram contains a horizontal chord. 
 \end{Lm}
 \begin{proof}
 It is the easiest to prove this fact visually by examining the structure of the face walk of a diagram. Start traversing the face walk from its topmost point, where the face walk goes along the arc $(4g) - (1)$. Since the diagram admits a reflection, if we start traversing the face walk in two opposite directions simultaneously with two `pointers' (Figure \ref{fig:face_walks}(a)), they will be drawing two lines which are mirror images of each other. These two lines must eventually meet, closing the face walk. That will happen as soon as they encounter either the arc $(2g) - (2g+1)$, or a horizontal chord. However, if they meet on a horizontal chord, then the arc $(2g) - (2g+1)$ is not covered by the only face walk, which is impossible. Hence the only face walk never encounters a horizontal chord and there are no horizontal chords in the diagram.
 \end{proof}
 
 \begin{Lm} 
 \label{lm:2}
 Any type I maximal diagram contains exactly one vertical and one horizontal chord. Removing these chords yields a type II maximal diagram. Conversely, inserting a horizontal and a vertical chord into a type II maximal diagram yields a type I maximal diagram.
 \end{Lm}
  \begin{proof}
 The existence of the vertical chord $(1)-(2g+1)$ is obvious. To prove that there is only one horizontal chord, we use the same visual argument as above. Start traversing the face walk from the arcs $(4g) - (1)$ and $(1) - (2)$ in two opposite directions (Figure \ref{fig:face_walks}(b)). The two lines being drawn by the traversal must eventually meet at a horizontal chord (Figure \ref{fig:face_walks}(c)). Now traverse the rest of the face walk starting from the initial points in the reverse direction (Figure \ref{fig:face_walks}(d)). The face walk will eventually close at some side of a horizontal chord. It must be the opposite side of the same horizontal chord, otherwise the face walk didn't pass through that side at all, contradicting diagram maximality.
 
 \begin{figure}[h]
\centering
\begin{tabular}[t]{cc}
	\begin{subfigure}[b]{0.22\textwidth}
	\centering
    		\includegraphics[scale=1.9]{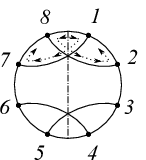}
 	\caption{Type I}
	\end{subfigure}

	\begin{subfigure}[b]{0.22\textwidth}
	\centering
    		\includegraphics[scale=1.9]{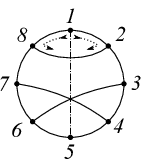}
 	\caption{Type II (step 1)}
	\end{subfigure}

	\begin{subfigure}[b]{0.22\textwidth}
	\centering
    		\includegraphics[scale=1.9]{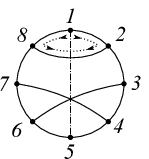}
 	\caption{Type II (step 2)}
	\end{subfigure}

	\begin{subfigure}[b]{0.22\textwidth}
	\centering
    		\includegraphics[scale=1.9]{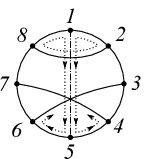}
 	\caption{Type II (step 3)}
	\end{subfigure}
\end{tabular}
\caption{Traversing face walks in maximal diagrams}
\label{fig:face_walks}
\end{figure}

 To prove the third part of the claim one can check that inserting a chord into a diagram in such a way that its two ends fall into the arcs belonging to the same face walk splits this walk into two. If such arcs belong to two different face walks, these face walks will conversely be merged. From the symmetry it follows that inserting a vertical chord into a type II maximal diagram splits its face walk into two walks that are mirror images of each other. Inserting any horizontal chord afterwards joins the arcs from two different face walks, yielding a type I maximal diagram.
  \end{proof} 
 
The Lemma \ref{lm:1} implies that the quotient map corresponding to a type II maximal chord diagram never contains a chord that lies on a boundary. This fact allows to prove the following theorem that reduces our enumeration problem to a problem that is solved in \cite{Ledoux} using the technique of matrix integrals.

 \begin{Th} 
 \label{th}
 There exists a bijection between type II maximal diagrams with $2g$ chords and rooted $g$-edge one-face one-vertex maps on orientable and non-orientable surfaces without a boundary. 
 \end{Th}
   \begin{proof}
First we replace the diagrams by their representations as $4g$-gons with edges identified pairwise. Since the diagrams in question are of type II, on a $4g$-gon the axis of symmetry passes through two vertices. By diagram maximality, after gluing the edges there is only one vertex left, hence the axis of symmetry crosses this only vertex, as well as the only face. For a quotient map this means that the surface has a boundary, and this boundary contains the only vertex and serves as an edge of the only face of the map. 

By Lemma 1 there are no horizontal chords in the diagram, hence the obtained surface has no other boundaries. In other words, the quotient map is built from an $(2g+1)$-gon with one edge (originally the axis of the reflection) being the boundary of the surface, and the other edges split into pairs and glued in one of two possible directions, depending on arrow orientation, which in its turn depends on the mutual position of a pair of chord in the chord diagram. Now choose the edge $(1) - (2)$ of the polygon to be the root of the map, contract the boundary into a point (or cap it with a disk) and remove the edge labelling. The resulting map is a rooted one-face one-vertex map with $g$ edges on an orientable or a non-orientable surface. 

This transformation is clearly invertible and its inverse can be applied to any one-face one-vertex rooted map as follows. Use the root to locate the `corner' to cut, and cut a hole in it to create a boundary. The obtained map on a surface with a boundary can be viewed as a $(2g+1)$-gon which has $2g$ sides glued pairwise in one or another direction and one more side which will become the axis of symmetry of the chord diagram. This $(2g+1)$-gon depicts one half of a symmetric $4g$-gon with its sides identified pairwise and a fixed reflection axis. Each pair of identified sides in the $(2g+1)$-gon stands for four sides of the $4g$-gon. These four sides are split into pairs in one of two possible symmetric ways depending on the direction of gluing of the sides of $(2g+1)$-gon. The chord diagram can be restored from this $4g$-gon in a unique way: each each pair of identified sides becomes a chord.
\end{proof}

 \section*{Enumerating formula}
 
By Theorem \ref{th} the number $d^{||}_{2g}$ is equal to the number of rooted maps with one face, one vertex and $g$ edges on both orientable and non-orientable surfaces. The recursion formula for this sequence can be obtained by substituting $k=1$ into \cite[Corollary 7]{Ledoux}:
$$
d^{||}_{2g} = \frac{1}{g+1} \Biggl( -(4g-1) d^{||}_{2g-2} + g(2g-3)(10g-9)d^{||}_{2g-4} + 30\BCf{2g-3}{3}d^{||}_{2g-6} - 240\BCf{2g-3}{5} d^{||}_{2g-8} \Biggr ).
$$
For this special case the only necessary initial conditions are $d^{||}_{2g} = 0$ for $g < 0$; $d^{||}_{0} = d^{||}_{2} = 1$. From Lemma \ref{lm:2} it follows that
$$
d^{|}_{2g} = (2g - 1) \, d^{||}_{2g-2},
$$
where the multiplier $(2g - 1)$ counts the number of ways to insert a horizontal chord into a type II diagram. By substituting these expressions and the expression for $d^*_{2g}$ \cite[Proposition 6.3]{Cori_Marcus}

$$
d^*_{2g} = \frac{1}{4g}\Bigl[ \frac{(4g)!}{4^g(2g+1)!} +\sum_{qk=4g, \, 2 \mid q}\varphi(q)\sum_{\gamma=0}^{k/4}\BCf{k}{4\gamma} \frac{(4\gamma)!}{4^\gamma(2\gamma+1)!}q^{2\gamma} + \sum_{qk=4g, \, 2 \nmid q}\varphi(q)q^{k/2} \frac{k!}{2^{k/2}(k/2+1)!}\Bigr]
$$
into the formula \eqref{final} the author computed the numbers $d^\circ_{2g}$ of non-isomorphic genus $g$ maximal chord diagrams (Table \ref{table}).

\begin{table}[h!]
\footnotesize
\centering
\begin{tabular}{r|r|r|r|r}
\midrule
$g$  & $d^*_{2g}$ & $d^|_{2g}$ & $d^{||}_{2g}$ & $d^\circ_{2g}$ \\ 
\midrule
1 & 1 & 1 & 1 & 1 \\
2 & 4 & 3 & 5 & 4 \\
3 & 131 & 25 & 41 & 82 \\
4 & 14118 & 287 & 509 & 7258 \\
5 & 2976853 & 4581 & 8229 & 1491629 \\
6 & 1013582110 & 90519 & 166377 & 506855279 \\
7 & 508233789579 & 2162901 & 4016613 & 254118439668 \\
8 & 352755124921122 & 60249195 & 113044185 & 176377605783906 \\
9 & 324039613564554401 & 1921751145 & 3630535785 & 162019808170348933 \\
10 & 380751174738424280720 & 68980179915 & 131095612845 & 190375587419231088550 \\
11 & 557175918657122229139987 & 2753007869745 & 5256401729985 & 278587959330563466969926 \\
12 & 993806827312044893602464496 & 120897239789655 & 231748716159765 & 496903413656110608290219603 \\
\midrule
\end{tabular}
\caption{The numbers of maximal chord diagrams by genus}
\label{table}
\end{table}

 \section*{Conclusion}

It turns out that maximal diagrams possess a structure which allows to derive simple formulas for enumerating them in labelled and unlabelled cases. Both the present result and the result of Cori and Marcus rely on some fact that a  quotient is maximal if and only if the original diagram is maximal. If we take one step further and try to enumerate unlabelled genus $g$ chord diagrams with a fixed number of faces greater than one, the connection between the numbers of faces on a diagram and its quotient will become much more complicated. At the moment the author is not aware of a way to obtain a formula for this more general problem.

\section*{References}

\bibliography{bib}

\end{document}